\documentclass[10pt,twoside]{article}
\usepackage{mathrsfs}
\usepackage{amssymb}
\usepackage{amsmath,color}
\usepackage[all]{xy}
\usepackage{amsthm}
\usepackage{url}

\usepackage{hyperref}
\hypersetup{colorlinks,linkcolor=blue,citecolor=blue}     %%% 为引用插入超链接

\usepackage{tikz}
\usepackage{stmaryrd}
\usepackage{xcolor}
\usetikzlibrary{arrows,calc}
\usepackage{etex}

\numberwithin{equation}{section}

\def\Ext{\mbox{\rm Ext}\,} \def\Hom{\mbox{\rm Hom}}

  \def\mod{\mbox{\rm \textbf{mod}}\,}

\def\cocone{\mbox{\rm cocone}}
\def\C{\mathcal {C}}
\def\A{\mathcal{A}} 
\def\Id{\mbox{\rm Id}\,} \def\Im{\mbox{\rm Im}\,}
\def\id{\mbox{\rm id}\,}

\def\C{\mathcal{C}}
\def\E{\mathbb{E}}

\def\del{\delta}

\newcommand{\res}{\operatorname{res.dim}}

\newcommand{\bsm}{\begin{smallmatrix}}
\newcommand{\esm}{\end{smallmatrix}} %%%%%%\bsm

\usepackage{paralist}

\newtheorem{theorem}{Theorem}[section]
\newtheorem{proposition}[theorem]{Proposition}
\newtheorem{definition}[theorem]{Definition}
\newtheorem{remark}[theorem]{Remark}
\newtheorem{lemma}[theorem]{Lemma}
\newtheorem{example}[theorem]{Example}

\newtheorem{corollary}[theorem]{Corollary}
\newtheorem{condition}[theorem]{Condition}
\newtheorem{theorem*}{Theorem}

\newcommand{\add}{\operatorname{add}}

\newcommand{\gl}{\operatorname{gl.dim}}

\newcommand{\pd}{\operatorname{pd}}
\newcommand{\ra}{\rightarrow}

\newcommand{\Ker}{\operatorname{Ker}}
\newcommand{\Coker}{\operatorname{Coker}}

\newcommand{\RNum}[1]{\uppercase\expandafter{\romannumeral #1\relax}}% 罗马数字

\title{ \bf Resolving subcategories and  dimensions in recollements of extriangulated categories  
}
\vspace{0.2cm}
\author{Xin Ma$^{1}$, Tiwei Zhao$^{2, }$\thanks{Corresponding author: tiweizhao@qfnu.edu.cn}, Xin Zhuang$^{1}$
\\
{\it \footnotesize 1. College of Science, Henan University of Engineering, Zhengzhou 451191, P.R. China}
\\
{\it \footnotesize 2. School of Mathematical Sciences, Qufu Normal University, Qufu 273165, P.R. China}
}
\date{ }
\begin{document}

\baselineskip=16pt
\maketitle

\begin{abstract}
Recently, Wang, Wei and Zhang introduced the notion of recollements of extriangulated categories. In this paper, let $(\mathcal{A},\mathcal{B},\mathcal{C})$ be a recollement of extriangulated categories.
We provide some methods to construct resolving subcategories in $(\mathcal{A},\mathcal{B},\mathcal{C})$. As applications of  the Auslander-Reiten correspondence,  we get the gluing of cotilting modules in a recollement of module categories for artin algebras.
We also give some bounds of resolution dimensions of the categories involved in $(\mathcal{A},\mathcal{B},\mathcal{C})$ with respect to resolving subcategories, which generalize some known results.
\end{abstract}

\pagestyle{myheadings}
\markboth{\rightline {\scriptsize X. Ma, T. Zhao, X. Zhuang}}%%%%%%%%%%%    %%%%%%%%    %%%%%%%%%    %%%%%%%页上角
{\leftline{\scriptsize Resolving subcategories and  dimensions in recollements of extriangulated categories}}%%%%%%%   %%%%%    %%%%%%    %%%%%%%页上角

%\section{Introduction} %delete * to number this section

%\tableofcontents

\section{Introduction}
Abelian categories and triangulated categories are two fundamental structures in algebra and geometry.
Nakaoka
and Palu \cite{Na} introduced the notion of extriangulated categories which is extracting properties on triangulated categories and exact categories (in particular, abelian categories).
%There exist extriangulated categories which are not exact
%nor triangulated, many examples can be founded in \cite{HZZ20P,INY18A, Na}.
Many
results given in exact categories and triangulated categories can be unified in the same
framework (see \cite{GNP21,HZZ21G,HZZ20P,HZZ21P,HZ21R,INY18A,Liu,ZZ20T,ZZ21G}).
Recollements of triangulated categories and abelian categories were introduced by Be{\u\i}linson, Bernstein and Deligne \cite{BBD},
which play an
important role in algebraic geometry and representation theory.
%, for instance \cite{BBD,BARI07H,CEPBSL86D,CEPBSL88F,KNJ94G}.
They are closely related each other, and possess similar properties in many aspects.
In order to give a simultaneous generalization of recollements of abelian categories and triangulated categories,
Wang, Wei and Zhang \cite{WWZ20R} introduced the notion of recollements of extriangulated categories.

Recently, gluing techniques with respect to a recollement of extriangulated categories have been investigated; for instance, He, Hu and Zhou \cite{HHZ21T} glued torsion pairs in a recollement of extriangulated categories;
Liu and Zhou \cite{LZ} glued cotorsion pairs in a recollement of extriangulated categories; etc.
These results give a simultaneous generalization of recollements of abelian categories and triangulated categories.

%also shows a new phenomena in the recollement of abelian categories.

Resolving subcategories play an important role in the study of triangulated categories and abelian categories.
In abelian categories, resolving subcategories and resolution dimensions are closely related with tilting theory (see \cite{AR}) and some homological conjectures (see \cite{H14}).
Resolving subcategories and resolution dimensions are also crucial to study the relative homological theory in triangulated categories (see \cite{MXResolving, MZ21R}).
In the situation of recollements,
 many authors studied the (relative) homological dimension in a recollement of abelian categories or triangulated categories (see \cite{GMT,HYG,PC14H,ZHJ,ZHJ1, ZJL}).

In this paper, we will provide some methods to construct  resolving subcategories, and investigate the relation of  resolution dimension with respect to resolving subcategories in a recollement of extriangulated categories.
The paper is organized as follows.

In Section 2, we summarize some basic definitions and properties of extriangulated categories, which will be used in this sequel.
In Section 3, we first recall the definition of resolving subcategories of extriangulated categories.
Latter, we give some methods to construct resolving subcategories in a recollement of extriangulated categories.
As applications, using the Auslander-Reiten correspondence (\cite{AR}),  we get the gluing of cotilting modules in a recollement of module categories for artin algebras.
In Section 4, we give some bounds for resolution dimensions with respect to resolving subcategories of the categories involved in a recollement of extriangulated categories.
Finally, in Section 5, we give some examples to illustrate the obtained results.

Throughout this paper,
all subcategories are assumed to be full, additive and closed under isomorphisms.
Let $\mathcal{C}$ be an extriangulated category, and let $\mathcal{X}$ be a class of objects of $\mathcal{C}$. We use $\add \mathcal{X}$ to
denote the subcategory of $\mathcal{C}$ consisting of
direct summands of finite direct sums of objects in $\mathcal{X}$.
Let $A$ be an artin algebra. We use $\mod A$ to denote the category of finitely generated left $A$-modules.
\section{Preliminaries}
Let us briefly recall some definitions and basic properties of extriangulated categories from \cite{Na}.
We omit some details here, but the reader can find
them in \cite{Na}.

Let $\mathcal{C}$ be an additive category equipped with an additive bifunctor
$$\mathbb{E}: \mathcal{C}^{\rm op}\times \mathcal{C}\rightarrow {\rm Ab},$$
where ${\rm Ab}$ is the category of abelian groups. For any objects $A, C\in\mathcal{C}$, an element $\delta\in \mathbb{E}(C,A)$ is called an \emph{$\mathbb{E}$-extension}.
Let $\mathfrak{s}$ be a correspondence which associates an equivalence class $$\mathfrak{s}(\delta)=\xymatrix@C=0.8cm{[A\ar[r]^x
 &B\ar[r]^y&C]}$$ to any $\mathbb{E}$-extension $\delta\in\mathbb{E}(C, A)$. This $\mathfrak{s}$ is called a {\it realization} of $\mathbb{E}$, if it makes the diagrams in \cite[Definition 2.9]{Na} commutative.
 A triplet $(\mathcal{C}, \mathbb{E}, \mathfrak{s})$ is called an {\it extriangulated category} if it satisfies the following conditions.
\begin{itemize}
\item[(1)] $\mathbb{E}\colon\mathcal{C}^{\rm op}\times \mathcal{C}\rightarrow \rm{Ab}$ is an additive bifunctor.
\item[(2)] $\mathfrak{s}$ is an additive realization of $\mathbb{E}$.
\item[(3)] $\mathbb{E}$ and $\mathfrak{s}$  satisfy the compatibility conditions in \cite[Definition 2.12]{Na}.
 \end{itemize}

Many examples can be founded in \cite{HZZ20P,INY18A,Na}.
%%Throughout, let $(\mathcal{C},\mathbb{E},\s)$ be an extriangulated category (see \cite{Na} for details).
%%$k$ be a field and $\mathcal{C}$ be a Krull-Schmidt Hom-finite, $k$-linear additive category.

%Throughout, let $k$ be a field, and Let $\mathcal{C}=(\C,\E,\s)$ be a Krull-Schmidt Hom-finite, $k$-linear extriangulated category, which satisfies some conditions $\rm (ET1)-(ET4),(ET3)^{op},(ET4)^{op}$ (see \cite{Na} for details).
%Some basic notions and basic properties of extriangulated categories can be see in \cite{Na}.
%We use the following notations (see \cite{Na})
%We use the following notations (see \cite{Na})
\begin{example}\label{example-extri}
(1) Exact categories, triangulated categories and extension-closed subcategories of triangulated categories are extriangulated categories.
Extension-closed subcategories of an extriangulated category are also extriangulated categories (see \cite[Remark 2.18]{Na}).

(2) If $\mathcal{C}$ is a triangulated category with suspension functor $\Sigma$ and
$\xi$ is a proper class of triangles (see \cite{B00R} for details), then $(\mathcal{C},\mathbb{E}_{\xi},\mathfrak{s}_{\xi})$ is an extriangulated category (see \cite[Remark 3.3]{HZZ20P}).

%(3) Let $\mathcal{C}$ be an extriangulated categories and $\mathcal{I}$ a subcategory of $\mathcal{C}$.
%If $\mathcal{I}\subseteq \mathcal{P(C)}\cap \mathcal{I(C)}$, then the ideal quotient
%$\mathcal{C}/\mathcal{I}$ has the structure of an extriangulated category.
\end{example}

We recall the following notations from \cite{Na}.

\begin{itemize}
\item[(1)] A sequence $\xymatrix@C=15pt{A\ar[r]^{x} & B \ar[r]^{y} & C}$ is called a {\it conflation} if it realizes some $\E$-extension $\del\in\E(C,A)$.
In this case, $\xymatrix@C=15pt{A\ar[r]^{x} & B}$ is called an {\it inflation} and $\xymatrix@C=15pt{B \ar[r]^{y} & C}$ is called a {\it deflation}. We call $\xymatrix@C=15pt{A\ar[r]^{x} & B \ar[r]^{y} & C\ar@{-->}[r]^{\del}&}$ an \emph{$\E$-triangle}.

\item[(2)]
Let $\mathcal{X}$ be a subcategory of $\mathcal{C}$, and let $\xymatrix@C=15pt{A\ar[r]^{x} & B \ar[r]^{y} & C\ar@{-->}[r]^{\del}&}$
be any $\mathbb{E}$-triangle.
\begin{itemize}
\item[(i)] We call $A$ the {\it cocone} of $y\xymatrix@C=15pt{\colon B\ar[r]& C},$ and denote it by ${\rm cocone}(y);$ we call $C$ the {\it cone} of $x\colon \xymatrix@C=15pt{A\ar[r]& B},$ and denote it by ${\rm cone}(x)$.
  \item[(ii)]  $\mathcal{X}$ is {\it closed under extensions} if $A, C\in \mathcal{X}$, it holds that $B\in \mathcal{X}$.
  \item[(iii)] $\mathcal{X}$ is {\it closed under cocones} (resp., {\it cones}) if $B, C\in \mathcal{X}$ (resp., $A,B\in \mathcal{X}$), it holds that $A\in \mathcal{X}$ (resp., $C\in \mathcal{X}$).
\end{itemize}
\end{itemize}

Throughout this paper,
for an extriangulated category $\mathcal{C}$, we assume the following condition, which is analogous
to the weak idempotent completeness (see \cite[Proposition 7.6]{Bu}).

\begin{condition}\label{WIC} {\rm(WIC) (see \cite[Condition 5.8]{Na})} Let $f:X\rightarrow Y$ and  $g:Y\rightarrow Z$ be any composable pair of morphisms in $\mathcal{C}$.
\begin{itemize}
\item[(1)]  If $gf$ is an inflation, then $f$ is an inflation.
\item[(2)] If $gf$ is a deflation, then $g$ is a deflation.
    \end{itemize}
\end{condition}

\begin{definition}{\rm(\cite[Definitions 3.23 and 3.25]{Na})
Let $\mathcal{C}$ be an extriangulated category.
\begin{itemize}
\item[(1)] An object $P$ in $\mathcal{C}$ is called {\em projective} if for any $\mathbb{E}$-triangle $A\stackrel{x}{\longrightarrow}B\stackrel{y}{\longrightarrow}C\stackrel{}\dashrightarrow$ and any morphism $c$ in $\mathcal{C}(P,C)$, there exists $b$ in $\mathcal{C}(P,B)$ such that $yb=c$.
We denote the full subcategory of projective objects in $\mathcal{C}$ by $\mathcal{P}(\mathcal{C})$.
Dually, the {\em injective} objects are defined, and the full subcategory of injective objects in $\mathcal{C}$ is denoted by $\mathcal{I}(\mathcal{C})$.
\item[(2)] We say that $\mathcal{C}$ {\em has enough projectives} if for any object $M\in\mathcal{C}$, there exists an $\mathbb{E}$-triangle $A\stackrel{}{\longrightarrow}P\stackrel{}{\longrightarrow}M\stackrel{}\dashrightarrow$ satisfying $P\in\mathcal{P}(\mathcal{C})$. Dually, we define that $\mathcal{C}$ {\em has enough injectives}.
\end{itemize}}
\end{definition}

\begin{remark}
$\mathcal{P(C)}$ is closed under direct summands, extensions and cocones. Dually, $\mathcal{I(C)}$ is closed under direct summands, extensions and cones.
\end{remark}

In extriangulated categories,
the notions of the left (right) exact sequences (resp., functor) and compatible morphisms can be founded in \cite[Definitions 2.9 and 2.12]{WWZ20R} for details.
Also the notion of  extriangulated (resp., exact) functor between two extriangulated categories can be found in \cite{Ben} (resp., \cite[Definition 2.13]{WWZ20R}).

Now we recall the concept of recollements of extriangulated categories \cite{WWZ20R}, which gives a simultaneous generalization of recollements of triangulated categories and abelian categories (see \cite{BBD, Fr}).

\begin{definition}\label{def-rec}{\rm(\cite[Definition 3.1]{WWZ20R}) }
Let $\mathcal{A}$, $\mathcal{B}$ and $\mathcal{C}$ be three extriangulated categories. A \emph{recollement} of $\mathcal{B}$ relative to
$\mathcal{A}$ and $\mathcal{C}$, denoted by ($\mathcal{A}$, $\mathcal{B}$, $\mathcal{C}$), is a diagram
\begin{equation}\label{recolle}
  \xymatrix{\mathcal{A}\ar[rr]|{i_{*}}&&\ar@/_1pc/[ll]|{i^{*}}\ar@/^1pc/[ll]|{i^{!}}\mathcal{B}
\ar[rr]|{j^{\ast}}&&\ar@/_1pc/[ll]|{j_{!}}\ar@/^1pc/[ll]|{j_{\ast}}\mathcal{C}}
\end{equation}
given by two exact functors $i_{*},j^{\ast}$, two right exact functors $i^{\ast}$, $j_!$ and two left exact functors $i^{!}$, $j_\ast$, which satisfies the following conditions:
\begin{itemize}
  \item [(R1)] $(i^{*}, i_{\ast}, i^{!})$ and $(j_!, j^\ast, j_\ast)$ are adjoint triples.
  \item [(R2)] $\Im i_{\ast}=\Ker j^{\ast}$.
  \item [(R3)] $i_\ast$, $j_!$ and $j_\ast$ are fully faithful.
  \item [(R4)] For each $B\in\mathcal{B}$, there exists a left exact $\mathbb{E}$-triangle sequence
$$
  \xymatrix{i_\ast i^!( B)\ar[r]^-{\theta_B}&B\ar[r]^-{\vartheta_B}&j_\ast j^\ast (B)\ar[r]&i_\ast (A)}
$$
 in $\mathcal{B}$ with $A\in \mathcal{A}$, where $\theta_B$ and  $\vartheta_B$ are given by the adjunction morphisms.
  \item [(R5)] For each $B\in\mathcal{B}$, there exists a right exact $\mathbb{E}$-triangle sequence
$$
  \xymatrix{i_\ast\ar[r]( A') &j_! j^\ast (B)\ar[r]^-{\upsilon_B}&B\ar[r]^-{\nu_B}&i_\ast i^\ast (B)&}
$$
in $\mathcal{B}$ with $A'\in \mathcal{A}$, where $\upsilon_B$ and $\nu_B$ are given by the adjunction morphisms.
\end{itemize}
\end{definition}

We collect some properties of recollements of extriangulated categories (see \cite{WWZ20R}).

\begin{lemma}\label{lem-rec} Let ($\mathcal{A}$, $\mathcal{B}$, $\mathcal{C}$) be a recollement of extriangulated categories.

$(1)$ All the natural transformations
$$i^{\ast}i_{\ast}\Rightarrow\Id_{\A},~\Id_{\A}\Rightarrow i^{!}i_{\ast},~\Id_{\C}\Rightarrow j^{\ast}j_{!},~j^{\ast}j_{\ast}\Rightarrow\Id_{\C}$$
are natural isomorphisms.
Moreover, $i^{!}$, $i^{*}$ and $j^{*}$ are dense.

$(2)$ $i^{\ast}j_!=0$ and $i^{!}j_\ast=0$.

$(3)$ $i^{\ast}$ preserves projective objects and $i^{!}$ preserves injective objects.

$(3')$ $j_{!}$ preserves projective objects and $j_{\ast}$ preserves injective objects.

$(4)$ If $i^{!}$ (resp., $j_{\ast}$) is  exact, then $i_{\ast}$ (resp., $j^{\ast}$) preserves projective objects.

$(4')$ If $i^{\ast}$ (resp., $j_{!}$) is  exact, then $i_{\ast}$ (resp., $j^{\ast}$) preserves injective objects.

$(5)$ If $\mathcal{B}$ has enough projectives, then $\mathcal{A}$ has enough projectives and $\mathcal{P(A)}=\add i^{*}(\mathcal{P(B)})$;
In addition, if $j^{*}$ preserves projectives, then $\mathcal{C}$ has enough projectives and
$\mathcal{P(C)}=\add j^{*}(\mathcal{P(B)})$.

$(6)$ If $i^{!}$ is exact, then $j_{\ast}$ is exact.

$(6')$
If $i^{\ast}$ is exact, then $j_{!}$ is  exact.

$(7)$ If $i^{!}$ is exact, for each $B\in\mathcal{B}$, there is an $\mathbb{E}$-triangle
  \begin{equation*}\label{third}
  \xymatrix{i_\ast i^! (B)\ar[r]^-{\theta_B}&B\ar[r]^-{\vartheta_B}&j_\ast j^\ast (B)\ar@{-->}[r]&}
   \end{equation*}
 in $\mathcal{B}$ where $\theta_B$ and  $\vartheta_B$ are given by the adjunction morphisms.

$(7')$ If $i^{\ast}$ is exact, for each $B\in\mathcal{B}$, there is an $\mathbb{E}$-triangle
  \begin{equation*}\label{four}
  \xymatrix{ j_! j^\ast (B)\ar[r]^-{\upsilon_B}&B\ar[r]^-{\nu_B}&i_\ast i^\ast (B) \ar@{-->}[r]&}
   \end{equation*}
in $\mathcal{B}$ where $\upsilon_B$ and $\nu_B$ are given by the adjunction morphisms.
\end{lemma}

\section{Gluing resolving subcategories in a recollement}

Throughout this paper, we will always assume that all extriangulated categories admit enough projective objects and injective objects.   % and \ref{Pmonic}.

Now we recall the notion of resolving subcategory in an extriangulated category, which gives a simultaneous generalization of  resolving subcategories in an abelian category (see \cite{AB69S}) and a triangulated category with a proper class $\xi$ of triangles (see \cite{MZ21R}).

\begin{definition} {\rm (cf. \cite[Page 243]{ZZ20T})
Let $\mathcal{C}$ be an extriangulated category and $\mathcal{X}$ a subcategory of $\mathcal{C}$. Then $\mathcal{X}$ is called a
{\it resolving} subcategory of $\mathcal{C}$ if the following conditions are satisfied.
\begin{itemize}
\item[(1)] $\mathcal{P(C)}\subseteq\mathcal{X}$.
\item[(2)] $\mathcal{X}$ is closed under extensions.
\item[(3)] $\mathcal{X}$ is closed under cocones.
\end{itemize}}
\end{definition}

Clearly,
$\mathcal{C}$ and $\mathcal{P(C)}$ are trivial resolving subcategories of $\mathcal{C}$.

\begin{remark}
\begin{itemize}
\item[(1)] If $\mathcal{C}$ is an abelian category, the resolving subcategory defined  as above coincide
with the earlier one given by Zhu in \cite{ZXS13R}.
\item[(2)]
If $\mathcal{C}$ is a triangulated category with a proper class $\xi$ of triangles, the resolving subcategory defined as above coincide
with the earlier one given by Ma and Zhao in \cite{MZ21R}
\end{itemize}
\end{remark}

From now on, let $(\mathcal{A},\mathcal{B},\mathcal{C})$ be a recollement of extriangulated categories.
The following result gives a method to glue a resolving subcategory in $\mathcal{B}$ from resolving subcategories in $\mathcal{A}$ and $\mathcal{C}$.
\begin{theorem}\label{main-resolving}
Let $(\mathcal{A},\mathcal{B},\mathcal{C})$ be a recollement of extriangulated categories as the diagram (\ref{recolle}).
Assume that $\mathcal{X'}$ and $\mathcal{X''}$ are resolving subcategories of $\mathcal{A}$ and $\mathcal{C}$ respectively. If $i^{*}$ is exact and $j^{*}$ preserves projective objects, then
$$\mathcal{X}:=\{B\in \mathcal{B}\mid i^{*}(B)\in \mathcal{X'}\ \text{and} \ j^{*}(B)\in \mathcal{X''}\}$$
 is a resolving subcategory of $\mathcal{B}$.
 In particular, we have
 \begin{itemize}
 \item[(1)] $i_{*}(\mathcal{X'})\subseteq \mathcal{X}$ and $j_{!}(\mathcal{X''})\subseteq \mathcal{X}$.
 \item[(2)] $i^{*}(\mathcal{X})=\mathcal{X'}$ and $j^{*}(\mathcal{X})=\mathcal{X''}$.
 \end{itemize}
\end{theorem}
\begin{proof}
Since $i^{*}$ preserves projective objects by Lemma \ref{lem-rec} and $j^{*}$ preserves projective objects by assumption,
we have $\mathcal{P(B)}\subseteq \mathcal{X}$.

Let $\xymatrix@C=0.5cm{X\ar[r]&Y\ar[r]&Z\ar@{-->}[r]&}$ be an $\mathbb{E}$-triangle in $\mathcal{B}$.
Applying the exact functors $i^{*}$ and $j^{*}$ to the above $\mathbb{E}$-triangle yields the following $\mathbb{E}$-triangles
$$\xymatrix{i^{*}(X)\ar[r]&i^{*}(Y)\ar[r]&i^{*}(Z)\ar@{-->}[r]&}$$
and
$$\xymatrix{j^{*}(X)\ar[r]&j^{*}(Y)\ar[r]&j^{*}(Z)\ar@{-->}[r]&}.$$

If $X$ and $Z$ are in $\mathcal{X}$, then $i^{*}(X)\in \mathcal{X'}$, $i^{*}(Z)\in \mathcal{X'}$, $j^{*}(X)\in \mathcal{X''}$, $j^{*}(Z)\in \mathcal{X''}$.
Notice that $\mathcal{X'}$ and $\mathcal{X''}$ are closed under extensions, so $i^{*}(Y)\in \mathcal{X'}$ and $j^{*}(Y)\in \mathcal{X''}$, then $Y\in \mathcal{X}$ and $\mathcal{X}$ is closed under extensions.

If $Y$ and $Z$ are in $\mathcal{X}$, then $i^{*}(Y)\in \mathcal{X'}$, $i^{*}(Z)\in \mathcal{X'}$, $j^{*}(Y)\in \mathcal{X''}$, $j^{*}(Z)\in \mathcal{X''}$.
Notice that $\mathcal{X'}$ and $\mathcal{X''}$ are closed under cocones, so $i^{*}(X)\in \mathcal{X'}$ and $j^{*}(X)\in \mathcal{X''}$, then $X\in \mathcal{X}$ and $\mathcal{X}$ is closed under cocones.

Thus $\mathcal{X}$ is a resolving subcategory of $\mathcal{B}$.

In particular,

(1) Since $i^{*}i_{*}\cong \Id_{\mathcal{A}}$ by Lemma \ref{lem-rec} and $j^{*}i_{*}=0$ by assumption, $i_{*}(\mathcal{X'})\in \mathcal{X}$.
Since $j^{*}j_{!}\cong \Id_{\mathcal{C}}$ and $i^{*}j_{!}=0$ by Lemma \ref{lem-rec}, $j_{!}(\mathcal{X''})\in \mathcal{X}$.

(2) By (1), we have $\mathcal{X'}\cong i^{*}i_{*}(\mathcal{X'})\subseteq i^{*}(\mathcal{X})$.
Notice that $i^{*}(\mathcal{X})\subseteq \mathcal{X'}$ is obvious.
Thus $i^{*}(\mathcal{X})=\mathcal{X'}$.
Similarly, we have $j^{*}(\mathcal{X})=\mathcal{X''}$.
\end{proof}

On the other hand, the following result gives a method to construct resolving subcategories in $\mathcal{A}$ and $\mathcal{B}$ from a resolving subcategory in $\mathcal{C}$.

\begin{theorem}\label{main-resolving-conve}
Let $(\mathcal{A},\mathcal{B},\mathcal{C})$ be a recollement of extriangulated categories as the diagram (\ref{recolle}).
Assume that $\mathcal{X}$ is a resolving subcategory of $\mathcal{B}$. Set $\mathcal{X}'=\add i^{*}(\mathcal{X})$, $\mathcal{X}''=\add j^{*}(\mathcal{X})$. Then we have the following statements.
\begin{itemize}
\item[(1)] If $i_{*}(\mathcal{X}')\subseteq \mathcal{X}$, then $\mathcal{X}'$ is a resolving subcategory of $\mathcal{A}$.
\item[(2)] If $j_{!}(\mathcal{X}'')\subseteq \mathcal{X}$, $j_{!}$ is exact and $j^{*}$ preserves projectives, then $\mathcal{X}''$ is a resolving subcategory of $\mathcal{C}$.
\item[(3)] If $i_{*}(\mathcal{X}')\subseteq \mathcal{X}$, $j_{!}(\mathcal{X}'')\subseteq \mathcal{X}$ and  $i^{*}$ is exact, then
$$\mathcal{X}=\{B\in \mathcal{B}\mid i^{*}(B)\in \mathcal{X}'\ \text{and} \ j^{*}(B)\in \mathcal{X}''\}.$$
\end{itemize}
\end{theorem}
\begin{proof}
(1) Since $\mathcal{P(B)}\subseteq \mathcal{X}$ by assumption, we have $i^{*}(\mathcal{P(B)})\subseteq  i^{*}(\mathcal{X})$. Moreover,  by Lemma \ref{lem-rec}, $\mathcal{P(A)}=\add i^{*}(\mathcal{P(B)})$, thus $\mathcal{P(A)}\subseteq \mathcal{X}'$.

Let
$$\xymatrix{X\ar[r]&Y\ar[r]&Z\ar@{-->}[r]&}$$
be an $\mathbb{E}$-triangle in $\mathcal{A}$.
Applying the exact functor $i_{*}$ to the above $\mathbb{E}$-triangle yields the following $\mathbb{E}$-triangle
$$\xymatrix{i_{*}(X)\ar[r]&i_{*}(Y)\ar[r]&i_{*}(Z)\ar@{-->}[r]&}$$
in  $\mathcal{B}$.

If $X,Z\in \mathcal{X}'$, since $i_{*}(\mathcal{X}')\subseteq\mathcal{X}$ by assumption and $\mathcal{X}$ is closed under extensions,
$i_{*}(Y)\in \mathcal{X}$.
So $Y\cong i^{*}i_{*}(Y)\in i^{*}(\mathcal{X})$, and thus $Y\in\mathcal{X}'$.
Then $\mathcal{X}'$ is closed under extensions.

If $Y,Z\in \mathcal{X}'$,
since $i_{*}(\mathcal{X}')\subseteq\mathcal{X}$ by assumption and $\mathcal{X}$ is closed under cocones, $i_{*}(X)\in \mathcal{X}$.
So $X\cong i^{*}i_{*}(X)\in i^{*}(\mathcal{X})$, and thus $X\in\mathcal{X}'$.
Then $\mathcal{X}'$ is closed under cocones.

Thus $i^{*}(\mathcal{X})$ is a resolving subcategory of $\mathcal{A}$.

(2)
Since $j^{*}$ preserves projectives by assumption, we have $\mathcal{P(C)}=\add j^{*}(\mathcal{P(B)})$ by Lemma \ref{lem-rec}.   Moreover, since $\mathcal{P(B)}\subseteq \mathcal{X}$, we have $j^{*}(\mathcal{P(B)})\subseteq j^{*}(\mathcal{X})$, and hence $\mathcal{P(C)}\subseteq \mathcal{X}''$.

Let
$$\xymatrix{X\ar[r]&Y\ar[r]&Z\ar@{-->}[r]&}$$
be an $\mathbb{E}$-triangle in $\mathcal{C}$.
Applying the exact functor $j_{!}$ to the above $\mathbb{E}$-triangle yields the following $\mathbb{E}$-triangle
$$\xymatrix{j_{!}(X)\ar[r]&j_{!}(Y)\ar[r]&j_{!}(Z)\ar@{-->}[r]&}$$
in $\mathcal{B}$.

Assume $X,Z\in \mathcal{X}''$. Since $j_{!}(\mathcal{X}'')\subseteq\mathcal{X}$ by assumption and $\mathcal{X}$ is closed under extensions,
$j_{!}(Y)\in \mathcal{X}$.
So $Y\cong j^{*}j_{!}(Y) \in j^{*}(\mathcal{X})$, and hence $Y\in \mathcal{X}''$.
Then $\mathcal{X}''$ is closed under extensions.

Assume $Y,Z\in \mathcal{X}''$.
Since $j_{!}(\mathcal{X}'')\subseteq\mathcal{X}$ by assumption and $\mathcal{X}$ is closed under cocones, $j_{!}(X)\in \mathcal{X}$.
So $X\cong j^{*}j_{!}(X) \in j^{*}(\mathcal{X})$, and hence $X\in\mathcal{X}''$.
Then $\mathcal{X}''$ is closed under cocones.

Thus $\mathcal{X}''$ is a resolving subcategory of $\mathcal{C}$.

(3) It is obvious that
 $$\mathcal{X}\subseteq \{B\in \mathcal{B}\mid i^{*}(B)\in \mathcal{X}'\ \text{and} \ j^{*}(X)\in \mathcal{X}''\}.$$
We only need to prove the contrary.
Let $B\in\mathcal{B}$ such that $ i^{*}(B)\in \mathcal{X}'$ and $ j^{*}(B)\in \mathcal{X}''$.
Since $i^{*}$ is exact by assumption, by Lemma \ref{lem-rec}, there exists an $\mathbb{E}$-triangle
$$\xymatrix@C=15pt{j_{!}j^{*}(B)\ar[r]&B\ar[r]&i_{*}i^{*}(B)
\ar@{-->}[r]&}$$
in $\mathcal{B}$.
Notice that $i_{*}i^{*}(B)\in i_{*}(\mathcal{X}')\subseteq \mathcal{X}$ and $j_{!}j^{*}(B)\in j_{!}(\mathcal{X}'')\subseteq\mathcal{X}$, so $B\in \mathcal{X}$ by the fact that $\mathcal{X}$ is closed under extensions.
Thus $$\{B\in \mathcal{B}\mid i^{*}(B)\in \mathcal{X}'\ \text{and} \ j^{*}(B)\in \mathcal{X}''\}\subseteq\mathcal{X}$$
as desired.
\end{proof}

Also, we have the following result.
\begin{proposition}\label{prop-conve}
Let $(\mathcal{A},\mathcal{B},\mathcal{C})$ be a recollement of extriangulated categories as the diagram (\ref{recolle}), and let
$\mathcal{X}$ be a resolving subcategory of $\mathcal{B}$. Set $\mathcal{X}''=\add j^{*}(\mathcal{X})$. If $j_{*}(\mathcal{X}'')\subseteq \mathcal{X}$ and $j_{*}$ is exact, then $\mathcal{X}''$ is a resolving subcategory of $\mathcal{C}$.
\end{proposition}
\begin{proof}
Since $j_{*}$ is exact by assumption, $j^{*}$ preserves projectives by Lemma \ref{lem-rec}, and hence $\mathcal{P(C)}=\add j^{*}(\mathcal{P(B)})$.
Notice that $\mathcal{P(B)}\subseteq \mathcal{X}$, so $ j^{*}(\mathcal{P(B)})\subseteq j^{*}(\mathcal{X})$, and thus $\mathcal{P(C)}\subseteq \mathcal{X}''$.

Let
$$\xymatrix{X\ar[r]&Y\ar[r]&Z\ar@{-->}[r]&}$$
be an $\mathbb{E}$-triangle in $\mathcal{C}$.
Applying the exact functor $j_{*}$ to the above $\mathbb{E}$-triangle yields the following $\mathbb{E}$-triangle
$$\xymatrix{j_{*}(X)\ar[r]&j_{*}(Y)\ar[r]&j_{*}(Z)\ar@{-->}[r]&}$$
in $\mathcal{B}$.

Assume $X,Z\in \mathcal{X}''$, since $j_{*}(\mathcal{X}'')\subseteq\mathcal{X}$ by assumption and $\mathcal{X}$ is closed under extensions,
$j_{*}(Y)\in \mathcal{X}$.
So $Y\cong j^{*}j_{*}(Y) \in j^{*}(\mathcal{X})$, and hence $Y\in\mathcal{X}''$.
Then $\mathcal{X}''$ is closed under extensions.

Assume $Y,Z\in \mathcal{X}''$,
since $j_{*}(\mathcal{X}'')\subseteq\mathcal{X}$ by assumption and $\mathcal{X}$ is closed under cocones, $j_{*}(X)\in \mathcal{X}$.
So $X\cong j^{*}j_{*}(X) \in j^{*}(\mathcal{X})$, and hence $X\in \mathcal{X}''$.
Then $\mathcal{X}''$ is closed under cocones.

Thus $\mathcal{X}''$ is a resolving subcategory of $\mathcal{C}$.
\end{proof}

At the end of this section, we give some applications.
We recall here the notion of cotilting modules for an artin algebra.
\begin{definition}{\rm (\cite{AR})
Let $A$ be an artin algebra and $T\in \mod A$. $T$ is called \emph{cotilting} if it satisfies the following conditions.
\begin{itemize}
\item[(1)] $T$ is selforthogonal, that is, $\Ext_{A}^{i}(T,T)=0$ for all $i>0$.
\item[(2)] ${\id_{A}}T<\infty$.
\item[(3)] For any injective $A$-module $I$, there exist some integer $n$ and an exact sequence
$$\xymatrix@C=15pt{0\ar[r]&T_{n}\ar[r]&T_{n-1}\ar[r]&\cdots\ar[r]&T_{1}\ar[r]&T_{0}\ar[r]&I\ar[r]&0}$$
in $\mod A$ with $T_{i}\in \add T$ for $0\leq i\leq n$.
\end{itemize}}
\end{definition}

We say a $A$-module $T$ is basic if in a direct sum decomposition into indecomposable modules, no indecomposable module appears more than once.
For a subclass $\mathcal{X}$ of $\mod A$, set
$${^{\perp}\mathcal{X}}:=\{A\in \mod A\mid \Ext_{A}^{i}(A,\mathcal{X})=0\text{ for all }i>0\}.$$
Dually, $\mathcal{X}^{\perp}$ is defined.

The following is the well-known Auslander-Reiten correspondence, which classifies cotilting modules using contravariantly finite resolving subcategories (see details for the notion of contravariantly finite subcategory in \cite[Page 114]{AR}).

\begin{theorem}\label{theorem-AR}
{\rm (\cite[Corollary 5.6]{AR})}
Let $A$ be an artin algebra of finite global dimension. Then the assignment $T\longrightarrow {^{\perp}T}$ gives a one-one correspondence between isomorphism classes of basic cotilting $A$-modules and contravariantly finite resolving subcategories of $\mod A$.
\end{theorem}
Following the above result,  for a contravariantly finite resolving subcategory $\mathcal{X}$ of $\mod A$, $T$ is a cotilting $A$-module satisfying $\add T=\mathcal{X}\cap\mathcal{X}^{\perp}$ (see {\rm (\cite[Theorem 5.5]{AR})}).
In this setting,
we say that there exists a cotilting $A$-module $T$ relative to the subcategory $\mathcal{X}:={^{\perp}T}$.

\begin{proposition}{\rm (\cite[Proposition 3.1]{ZhYY})}\label{Prop-contr}
Let $(\mathcal{A},\mathcal{B},\mathcal{C})$ be a recollement of abelian categories.
Assume that $\mathcal{X'}$ and $\mathcal{X''}$ are contravariantly finite subcategories of $\mathcal{A}$ and $\mathcal{C}$ respectively, and if $i^{*}$ is exact, then
$$\mathcal{X}:=\{B\in \mathcal{B}\mid i^{*}(B)\in \mathcal{X'}\ \text{and} \ j^{*}(B)\in \mathcal{X''}\}$$
 is a contravariantly finite subcategory of $\mathcal{B}$.
\end{proposition}

Combining Theorems \ref{main-resolving} and \ref{theorem-AR}, and Proposition \ref{Prop-contr}, we can get a gluing method of cotilting modules as follows.

\begin{proposition}\label{appl1}
Let $A$, $B$ and $C$ be artin algebras with finite global dimensions, and let
\begin{equation*}
  \xymatrix{\mod A\ar[rr]|{i_{*}}&&\ar@/_1pc/[ll]|{i^{*}}\ar@/^1pc/[ll]|{i^{!}}\mod B
\ar[rr]|{j^{\ast}}&&\ar@/_1pc/[ll]|{j_{!}}\ar@/^1pc/[ll]|{j_{\ast}}\mod C}
\end{equation*}
 be a recollement of module categories.
Assume that $T'$ and $T''$ are cotilting modules in $\mod A$ and $\mod C$ respectively.
If $i^{*}$ is exact and $j^{*}$ preserves projectives, then there exists a cotilting $B$-module $T$ glued by $T'$ and $T''$, that is, $T$ is a cotilting module relative to the subcategory
$$\mathcal{X}:=\{B\in \mod B\mid i^{*}(B)\in {^{\perp}T'}\ \text{and} \ j^{*}(B)\in {^{\perp}T''}\}$$
of $\mod B$ such that $\mathcal{X}={^{\perp}T}$.
\end{proposition}

\begin{remark}
In fact,
in Proposition \ref{appl1}, one can only assume that $A$ and $C$ have finite global dimensions. Then
by \cite[Theorem 4.4]{PC14H}, the global dimension of $B$ is also finite.
\end{remark}

Conversely, we have
\begin{proposition}\label{appl2}
Let $A$, $B$ and $C$ be artin algebras with finite global dimensions, and let
\begin{equation*}
  \xymatrix{\mod A\ar[rr]|{i_{*}}&&\ar@/_1pc/[ll]|{i^{*}}\ar@/^1pc/[ll]|{i^{!}}\mod B
\ar[rr]|{j^{\ast}}&&\ar@/_1pc/[ll]|{j_{!}}\ar@/^1pc/[ll]|{j_{\ast}}\mod C}
\end{equation*}
 be a recollement of module categories.
Assume that $T$ is a cotilting $B$-module. Let  $\mathcal{X}'=\add i^{*}({^\perp T})$ and $\mathcal{X}''=\add j^{*}({^\perp T})$.  Then we have the following statements.
\begin{itemize}
\item[(1)] If $i_{*}(\mathcal{X}')\subseteq {^\perp T}$, then there exists a cotilting $A$-module $T'$ relative to the subcategory $\mathcal{X}'$.
\item[(2)] If $j_{!}(\mathcal{X}'')\subseteq {^\perp T}$, $j_{!}$ is exact and $j^{*}$ preserves projectives, then there exists a cotilting $C$-module $T''$ relative to  the subcategory $\mathcal{X}''$.
\end{itemize}
\end{proposition}
\begin{proof}
Assume that $T$ is a cotilting $B$-module, then ${^{\perp}T}$ is a contravariantly finite resolving subcategory of $\mod B$ by Theorem \ref{theorem-AR}.
By \cite[Lemma 3.6]{MXZ22S}, $\mathcal{X}'$ and $\mathcal{X}''$ are contravariantly finite subcategories of $\mod A$ and $\mod C$ respectively.
By Theorem \ref{main-resolving-conve}(1), $\mathcal{X}'$ is a resolving subcategory of $\mod A$.
 By Theorem \ref{main-resolving-conve}(2), $\mathcal{X}''$ is a resolving subcategory of $\mod C$.
Then the assertions follow from Theorem \ref{theorem-AR}.
\end{proof}

\begin{proposition}\label{appl3}
Let $A$, $B$ and $C$ be artin algebras with finite global dimensions, and let
\begin{equation*}
  \xymatrix{\mod A\ar[rr]|{i_{*}}&&\ar@/_1pc/[ll]|{i^{*}}\ar@/^1pc/[ll]|{i^{!}}\mod B
\ar[rr]|{j^{\ast}}&&\ar@/_1pc/[ll]|{j_{!}}\ar@/^1pc/[ll]|{j_{\ast}}\mod C}
\end{equation*}
 be a recollement of module categories.
Assume that $T$ is a cotilting $B$-module. Let $\mathcal{X}''=\add j^{*}({^\perp T})$.
If $j_{*}(\mathcal{X}'')\subseteq {^\perp T}$ and $j_{*}$ is exact, then there exists a cotilting $C$-module $T''$ relative to  the subcategory $\mathcal{X}''$.
\end{proposition}
\begin{proof}
Assume that $T$ is a cotilting $B$-module, then ${^{\perp}T}$ is a contravariantly finite resolving subcategory of $\mod B$ by Theorem \ref{theorem-AR}.
By \cite[Lemma 3.6]{MXZ22S}, $\mathcal{X}''$ is a contravariantly finite subcategories of $\mod C$.
 By Theorem \ref{main-resolving-conve}(2), $\mathcal{X}''$ is a resolving subcategory of $\mod C$.
Then the assertion follows from Theorem \ref{theorem-AR}.
\end{proof}

%%%  Since $j^{*}$ preserves projectives, $\gl C\leq \gl B$

\section{Resolving resolution dimensions and recollements}

In this section, we give some bounds for resolution dimension of (resolving) subcategories  in a recollement of extriangulated categories.

 Recall from \cite{ZZ20T} that an {\em $\mathbb{E}$-triangle sequence} is defined as a sequence
 $$\cdots {\longrightarrow}X_{n+2}\stackrel{d_{n+2}}{\longrightarrow}X_{n+1}\stackrel{d_{n+1}}{\longrightarrow} X_{n}\stackrel{d_{n}}{\longrightarrow}X_{n-1}{\longrightarrow}\cdots $$
 in $\mathcal{C}$ such that for any $n$,
 there exist $\mathbb{E}$-triangles $K_{n+1}\stackrel{g_{n}}{\longrightarrow}X_{n}\stackrel{f_{n}}{\longrightarrow}K_{n}\stackrel{}\dashrightarrow$ and the differential $d_n=g_{n-1}f_n$.

 Now we introduce the notion of resolution dimension for a subcategory of $\mathcal{C}$.
 \begin{definition}
{\rm Let $\mathcal{X}$ be a subcategory of $\mathcal{C}$ and $M$ an object in $\mathcal{C}$.
The {$\mathcal{X}$-resolution dimension} of $M$, written $\mathcal{X}$-$\res M$, is defined by
\begin{align*}
\mathcal{X}\text{-}\res M=&\inf \{n \geq 0\mid\text{ there exists an } \mathbb{E}\text{-triangle sequence}\\
&\xymatrix@C=15pt
{0\ar[r]&X_{n}\ar[r]&\cdots\ar[r]&X_{1}\ar[r]&X_{0}\ar[r]&M\ar[r]&0} \text{ in } \mathcal{C} \text{ with } X_{i}\in \mathcal{X} \text{ for }0\leq i\leq n\}.
\end{align*}
The resolution dimension of $\mathcal{C}$, denoted by $\mathcal{X}\text{-}\res \mathcal{C}$, is defined as
$$\mathcal{X}\text{-}\res  \mathcal{C}:=\sup\{\mathcal{X}\text{-}\res M\mid M\in \mathcal{C}\}.$$}
\end{definition}

%For an $\mathbb{E}$-triangle sequence
%$$\xymatrix@C=15pt
%{\cdots\ar[r]^{f_{n+1}}&X_{n}\ar[r]&\cdots\ar[r]^{f_{2}}&X_{1}\ar[r]^{f_{1}}&X_{0}\ar[r]^{f_{0}}&M\ar[r]&0}$$ with all $X_{i}\in \mathcal{X}$.
%The $\cocone( f_{n-1})$ is called an $n$th $\mathcal{X}$-syzygy of $M$, denoted by $\Omega^{n}_{\mathcal{X}}(M)$.

In case for $\mathcal{X}=\mathcal{P(C)}$, $\mathcal{X}\text{-}\res M$ coincides with $\pd M$ (the projective dimension of $M$), and $\mathcal{X}\text{-}\res\mathcal{C}$ coincides with $\gl \mathcal{C}$ (the global dimension of $\mathcal{C}$) defined in \cite{GMT}.

Now,
we give some useful facts.

The following result gives a simultaneous generalization of  \cite[Lemma 2.2]{ZXS13R} for an abelian category and \cite[Proposition 3.4]{MZ21R} for a triangulated category with a proper class $\xi$ of triangles.

\begin{lemma} \label{prop-resdim}
Let $\mathcal{C}$ be an extriangulated categories, and let $\mathcal{X}$ be a resolving subcategory of $\mathcal{C}$.
Assume that
$$\xymatrix@C=15pt{X\ar[r]&Y\ar[r]&Z\ar@{-->}[r]&}$$ is an $\mathbb{E}$-triangle in $\mathcal{C}$.
Then we have the following statements.
\begin{itemize}
\item[(1)] $\mathcal{X}$-$\res Y\leq \max\{\mathcal{X}\text{-}\res X, \mathcal{X}\text{-}\res Z\}$.
\item[(2)] $\mathcal{X}$-$\res X\leq \max\{\mathcal{X}\text{-}\res Y, \mathcal{X}\text{-}\res Z-1\}$.
\item[(3)] $\mathcal{X}$-$\res Z\leq \max\{\mathcal{X}\text{-}\res X+1, \mathcal{X}\text{-}\res Y\}$.
    \end{itemize}
\end{lemma}
\begin{proof}
The proof is similar to \cite[Proposition 3.4]{MZ21R}.
\end{proof}

The following result gives a simultaneous generalization of  \cite[Lemma 2.1]{ZXS13R} for an abelian category and \cite[Proposition 3.2]{MZ21R} for a triangulated category with a proper class $\xi$ of triangles.

\begin{lemma}\label{omega}
Let $\mathcal{C}$ be an extriangulated category and $\mathcal{X}$ a resolving subcategory of $\mathcal{C}$.
For any object $M\in \mathcal{C}$, if
$$\xymatrix@C=15pt
{0\ar[r]&X_{n}\ar[r]&\cdots\ar[r]&X_{1}\ar[r]&X_{0}\ar[r]&M\ar[r]&0}$$ and
$$\xymatrix@C=15pt
{0\ar[r]&Y_{n}\ar[r]&\cdots\ar[r]&Y_{1}\ar[r]&Y_{0}\ar[r]&M\ar[r]&0}$$
are $\mathbb{E}$-triangle sequences with all $X_{i}$ and $Y_{i}$ in $\mathcal{X}$ for $0\leq i\leq n-1$, then $X_{n}\in\mathcal{X}$ if and only if $Y_{n}\in\mathcal{X}$.
\end{lemma}
\begin{proof}
The proof is similar to \cite[Proposition 3.2]{MZ21R}.
\end{proof}

\begin{lemma}\label{lem-j_{!}}
Let $(\mathcal{A},\mathcal{B},\mathcal{C})$ be a recollement of extriangulated categories  as the diagram (\ref{recolle}), and let $\mathcal{X_{B}}$ and $\mathcal{X_{C}}$ be resolving subcategories of $\mathcal{B}$ and
$\mathcal{C}$ respectively. Let $C\in \mathcal{C}$. Then
$$\mathcal{X_{B}}\text{-} \res j_{!}(C) \leq  \mathcal{X_{C}}\text{-}\res C+ \max\{\mathcal{X_{B}}\text{-}\res i_{*}(\mathcal{A}), \mathcal{X_{B}}\text{-}\res j_{!}(\mathcal{X_{C}})\} + 1.$$
\end{lemma}
\begin{proof}
If $\mathcal{X_{C}}\text{-}\res C=\infty$ or $\max\{\mathcal{X_{B}}\text{-}\res i_{*}(\mathcal{A}), \mathcal{X_{B}}\text{-}\res j_{!}(\mathcal{X_{C}})\}=\infty$,
there is
nothing to prove.
Assume that
$\max\{\mathcal{X_{B}}\text{-}\res i_{*}(\mathcal{A}), \mathcal{X_{B}}\text{-}\res j_{!}(\mathcal{X_{C}})\}=m$.
The proof will be proceed by induction on the $\mathcal{X_{C}}$-resolution dimension of $C$.
If $C\in \mathcal{X_{C}}$, the assertion holds obviously.
Now suppose that $\mathcal{X_{C}}\text{-}\res C=n\geq 1$.
By
Lemma \ref{omega}, we have the following $\mathbb{E}$-triangle sequence

$$\xymatrix@C=15pt{X_{n}\ar[rr]&&P_{n-1}\ar[rr]&&\cdots\ar[rr]&&P_{2}\ar[rr]&&P_{1}\ar[rr]\ar[rd]&&P_{0}\ar[rr]^{g}&&C\\
&&&&&&&&&K_{1}\ar[ru]_{f}&&&}$$
with $P_{i}\in \mathcal{P(C)}\subseteq \mathcal{X_C}$ for $0\leq i\leq n-1$ and $X_{n}\in \mathcal{X_C}$.

Notice that $\mathcal{X_{C}}\text{-}\res K_{1}\leq n-1$, by induction hypothesis, we have
$$\mathcal{X_{B}}\text{-}\res j_{!}(K_{1})\leq \mathcal{X_{C}}\text{-}\res K_{1}+m+1\leq n-1+m+1=m+n.$$
Since $j_{!}$ is right exact, there is an $\mathbb{E}$-triangle $K_{1}'\stackrel{h_{2}}{\longrightarrow}j_{!}(P_{0})\stackrel{j_{!}(g)}{\longrightarrow}j_{!}(C)\stackrel{}\dashrightarrow$
in $\mathcal{B}$ and a deflation $\xymatrix@C=15pt{h_{1}:j_{!}(K_{1})\ar[r]&K'_{1}}$ which is compatible, such that $j_{!}(f)=h_{2}h_{1}$.

Since $j^{*}j_{!}= {\rm Id}_{\mathcal{C}}$ by Lemma \ref{lem-rec},
$$\xymatrix@C=20pt{j^{*}j_{!}(K_{1})\ar[r]^{j^{*}j_{!}(f)}&j^{*}j_{!}(P_{0})\ar[r]^{j^{*}j_{!}(g)}&j^{*}j_{!}(C)\ar@{-->}[r]&}$$
is an $\mathbb{E}$-triangle in $\mathcal{C}$.
Since $f=j^{*}j_{!}(f)=(j^{*}(h_{2}))(j^{*}(h_{1}))$, so $j^{*}(h_{1})$ is an inflation by Condition \ref{WIC}.
Notice that $j^{*}(h_{1})$ is a deflation and compatible since $j^{*}$ is exact, so $j^{*}(h_{1})$ is an isomorphism.
Thus  $j^{*}j_{!}(K_{1})\cong j^{*}(K'_{1})$.
Set $K''_{1}=\cocone (h_{1})$, consider the following $\mathbb{E}$-triangle
\begin{align}\label{E-triangle-1}
\xymatrix@C=20pt{K''_{1}\ar[r]&j_{!}(K_{1})\ar[r]^{h_1}&K'_{1}\ar@{-->}[r]&}
\end{align}
in $\mathcal{B}$.
Since $j^{*}$ is exact,  $j^{*}(K''_{1})=0$. By (R2), there exists an object $A'\in \mathcal{A}$ such that $K''_{1}\cong i_{*}(A')$.
Then $\mathcal{X_{B}}\text{-}\res K''_{1}\leq m$ by assumption.
Apply Lemma \ref{prop-resdim} to the $\mathbb{E}$-triangle (\ref{E-triangle-1}),
we have $\mathcal{X_{B}}\text{-}\res K'_{1}\leq m+n$.
Notice that $j_{!}$ preserves projectives by Lemma \ref{lem-rec}, so $j_{!}(P_{0})\in \mathcal{P(B)}(\subseteq \mathcal{X}_{\mathcal{B}})$.
It follows that $\mathcal{X_{B}}\text{-}\res  j_{!}(C)\leq m+n+1$.
\end{proof}

Now we  give the main theorem.
\begin{theorem}\label{main-dim}
Let $(\mathcal{A},\mathcal{B},\mathcal{C})$ be a recollement of extriangulated categories as the diagram (\ref{recolle}), and assume that
$\mathcal{B}$ and $\mathcal{C}$ have enough projective objects. Let $\mathcal{X_{A}}$, $\mathcal{X_{B}}$ and $\mathcal{X_{C}}$ be resolving subcategories
of $\mathcal{A},\ \mathcal{B}$ and $\mathcal{C}$, respectively. Then we have the following statements.
\begin{itemize}
\item[(1)] $\mathcal{X_{B}}\text{-}\res \mathcal{B} \leq  \mathcal{X_{C}}\text{-}\res \mathcal{C}+\max\{\mathcal{X_{B}}\text{-}\res  i_{*}(\mathcal{A}), \mathcal{X_{B}}\text{-}\res  j_{!}(\mathcal{X_{C}})\} + 1.$
\item[(2)] $\mathcal{X_{B}}\text{-}\res i_{*}(\mathcal{A})\leq \mathcal{X_{A}}\text{-}\res \mathcal{A}+\mathcal{X_{B}}\text{-}\res i_{*}(\mathcal{X_{A}})$.

\item[(3)] $\mathcal{X_{B}}\text{-}\res \mathcal{B} \leq \mathcal{X_{A}}\text{-}\res \mathcal{A}+ \mathcal{X_{C}}\text{-}\res \mathcal{C}+\max\{\mathcal{X_{B}}\text{-}\res i_{*}(\mathcal{X_{A}}), \mathcal{X_{B}}\text{-}\res  j_{!}(\mathcal{X_{C}})\} + 1.$

\item[(4)] $\mathcal{X_{C}}\text{-}\res \mathcal{C}\leq \mathcal{X_{B}}\text{-}\res \mathcal{B}+ \mathcal{X_{C}}\text{-}\res  j^{*}(\mathcal{X_{B}})$.

\item[(5)] If $j_{!}(\mathcal{X_{C}})\subseteq \mathcal{X_{B}}$ or $j_{*}(\mathcal{X_{C}})\subseteq \mathcal{X_{B}}$, then
$$\mathcal{X_{B}}\text{-}\res \mathcal{B}\leq \mathcal{X_{C}}\text{-}\res \mathcal{C}+\mathcal{X_{B}}\text{-}\res i_{*}(\mathcal{A})+1$$
and
$$\mathcal{X_{B}}\text{-}\res \mathcal{B}\leq \mathcal{X_{C}}\text{-}\res \mathcal{C}+ \mathcal{X_{A}}\text{-}\res \mathcal{A}+\mathcal{X_{B}}\text{-}\res i_{*}(\mathcal{X_A})+1.$$

\item[(6)] If $i_{*}(\mathcal{X_{A}}) \subseteq \mathcal{X_{B}}$, then
$\mathcal{X_{B}}\text{-}\res i_{*}(\mathcal{A})\leq \mathcal{X_{A}}\text{-}\res \mathcal{A}$.

\item[(7)] If $j_{!}(\mathcal{X_{C}}) \subseteq \mathcal{X_{B}}$ (or $j_{*}(\mathcal{X_{C}}) \subseteq{\mathcal{X_{B}}}$)  and $i_{*}(\mathcal{X_{A}}) \subseteq\mathcal{X_{B}}$, then
$$\mathcal{X_{B}}\text{-}\res \mathcal{B}\leq \mathcal{X_{A}}\text{-}\res \mathcal{A}+\mathcal{X_{C}}\text{-}\res \mathcal{C}+1.$$

\item[(8)] If $j^{*}(\mathcal{X_{B}}) \subseteq \mathcal{X_{C}}$, then
$\mathcal{X_{C}}\text{-}\res \mathcal{C}\leq \mathcal{X_{B}}\text{-}\res \mathcal{B}$.

\end{itemize}

\end{theorem}

\begin{proof}

(1)
Suppose that
$\max\{\mathcal{X_{B}}\text{-}\res i_{*}(\mathcal{A}),\mathcal{X_{B}}\text{-}\res j_{!}(\mathcal{X_{C}})\} = m < \infty$ and $\mathcal{X_{C}}\text{-}\res \mathcal{C}= n <\infty$.
Let $B\in \mathcal{B}$.
By (R5),
there exists a commutative diagram
\begin{equation*}
\xymatrix{
  &i_{*}(A') \ar[r]&j_{!}j^{*}(B)\ar[rr]^-{\upsilon_B}\ar[dr]_{h_{2}}&  &B\ar[r]^-{\nu_B}&i_{*}i^{*}(B) &\\
           &                &       &  B' \ar[ur]_{h_{1}}& }
\end{equation*}
in $\mathcal{B}$ such that $i_{*}(A')\stackrel{}{\longrightarrow}j_{!}j^{*}(B)\stackrel{h_{2}}{\longrightarrow}B'\stackrel{}\dashrightarrow$ and $B'\stackrel{h_{1}}{\longrightarrow}B{\longrightarrow}i_{*}i^{*}(B)\stackrel{}\dashrightarrow$ are $\mathbb{E}$-triangles and $h_{2}$ is compatible.
Notice that
$\mathcal{X_{B}}\text{-}\res i_{*}(A')\leq m$ and $\mathcal{X_{B}}\text{-}\res i_{*}i^{*}(B)\leq m$.
By Lemmas \ref{prop-resdim} and \ref{lem-j_{!}},
\begin{align*}
\mathcal{X_{B}}\text{-}\res B&\leq \max\{\mathcal{X_{B}}\text{-}\res B',\mathcal{X_{B}}\text{-}\res i_{*}i^{*}(B)\}\\
&\leq \max\{\mathcal{X_{B}}\text{-}\res i_{*}(A')+1,\mathcal{X_{B}}\text{-}\res j_{!}j^{*}(B),\mathcal{X_{B}}\text{-}\res i_{*}i^{*}(B)\}\\
&\leq \max\{m+1,\mathcal{X_{C}}\text{-}\res j^{*}(B)+m+1,m\}
\end{align*}
Notice that $ \mathcal{X_{C}}\text{-}\res j^{*}(B)\leq n$, so $\mathcal{X_{B}}\text{-}\res B\leq m+n+1$.

(2)
 Suppose that $\mathcal{X_{B}}\text{-}\res i_{*}(\mathcal{X_{A}})=n<\infty$ and $\mathcal{X_{A}}\text{-}\res \mathcal{A}=m< \infty$.
Let $A\in \mathcal{A}$.
If $A\in \mathcal{X_{A}}$, then $\mathcal{X_{B}}\text{-}\res  i_{*}(A)\leq n$ and our result holds.
Now suppose that $\mathcal{X_{A}}\text{-}\res A=s\leq m$.
Consider the following $\mathbb{E}$-triangle sequence
$$\xymatrix{X'_{s}\ar[r]&X'_{s-1}\ar[r]&\cdots\ar[r]&X'_{1}\ar[r]&X'_{0}\ar[r]&A}$$
in $\mathcal{A}$ with $X'_{i}\in \mathcal{X_{A}}$ for $0\leq i\leq s$.
Since $i_{*}$ is exact,
$$\xymatrix{i_{*}(X'_{s})\ar[r]&i_{*}(X'_{s-1})\ar[r]&\cdots\ar[r]&i_{*}(X'_{1})\ar[r]&i_{*}(X'_{0})\ar[r]&i_{*}(A)}$$
is an $\mathbb{E}$-triangle sequence in $\mathcal{B}$.
Notice that $\mathcal{X_{B}}\text{-}\res  i_{*}(X'_{i})\leq n$ by assumption,
so $\mathcal{X_{B}}\text{-}\res  i_{*}(A)\leq s+n \leq m+n$ by Lemma \ref{prop-resdim}.

(3) It follows from  (1) and (2).

(4)
Suppose that $\mathcal{X_{B}}\text{-}\res \mathcal{B}= m<\infty$ and $\mathcal{X_{C}}\text{-}\res j^{*}(\mathcal{X_{B}})=n<\infty$.
For any object $C\in \mathcal{C}$, $j_{!}(C)\in \mathcal{B}$. Assume that $\mathcal{X_{B}}\text{-}\res j_!(C)=s\leq m$,
and consider the following $\mathbb{E}$-triangle sequence
$$\xymatrix{X_{s}\ar[r]&X_{s-1}\ar[r]&\cdots\ar[r]&X_{1}\ar[r]&X_{0}\ar[r]&j_{!}(C)}$$
in $\mathcal{B}$ with $X_{i}\in \mathcal{X_B}$ for $0\leq i\leq s$.

Since $j^{*}$ is exact,
$$\xymatrix{j^{*}(X_{s})\ar[r]&j^{*}(X_{s-1})\ar[r]&\cdots\ar[r]&j^{*}(X_{1})\ar[r]&j^{*}(X_{0})\ar[r]&C(\cong j^{*}j_{!}(C))}$$
is an $\mathbb{E}$-triangle sequence in $\mathcal{C}$.
Notice that $\mathcal{X_C}\text{-}\res j^{*}(X_{i})\leq n$ by assumption,
so $\mathcal{X_C}\text{-}\res C\leq s+n \leq m+n$ by Lemma \ref{prop-resdim}.

(5)   If $j_{!}(\mathcal{X_{C}})\subseteq \mathcal{X_{B}}$, then $\mathcal{X_{B}}\text{-}\res j_{!}(\mathcal{X_{C}})=0$.
 So by (1)
 $$\mathcal{X_{B}}\text{-}\res \mathcal{B}\leq \mathcal{X_{C}}\text{-}\res \mathcal{C}+\mathcal{X_{B}}\text{-}\res i_{*}(\mathcal{A})+1.$$

If $j_{*}(\mathcal{X_{C}})\subseteq \mathcal{X_{B}}$, for every object $M\in \mathcal{X_{C}}$, $j_{!}(M)\in  \mathcal{B}$.
By (R4),
there exists a commutative diagram
\begin{equation*}
\xymatrix{
  &i_{*}i^{!}(j_{!}(M)) \ar[r]^-{\theta_{j_{!}(M)}}&j_{!}(M)\ar[rr]^-{\vartheta_{j_{!}(M)}}\ar[dr]_{h'_{2}}&  &j_{*}j^{*}(j_{!}(M))\ar[r]^-{h'}&i_{*}(A'') &\\
           &                &       &  B'' \ar[ur]_{h'_{1}}& }
\end{equation*}
in $\mathcal{B}$ such that $i_{*}i^{!}(j_{!}(M))\stackrel{}{\longrightarrow}j_{!}(M)\stackrel{h'_{2}}{\longrightarrow}B''\stackrel{}\dashrightarrow$ and $B''\stackrel{h'_{1}}{\longrightarrow}j_{*}j^{*}(j_{!}(M))\stackrel{h'}{\longrightarrow}i_{*}(A'')\stackrel{}\dashrightarrow$ are $\mathbb{E}$-triangles and $h'_{1}$ is compatible.
Notice that $j_{*}j^{*}(j_{!}(M))\cong j_{*}(M)\in \mathcal{X_{B}}$ by Lemma \ref{lem-rec} and assumption, so $\mathcal{X_{B}}\text{-}\res j_{*}j^{*}(j_{!}(M))=0$.
Then by Lemma \ref{prop-resdim},
\begin{align*}
\mathcal{X_{B}}\text{-}\res j_{!}(M)&\leq \max\{ \mathcal{X_{B}}\text{-}\res i_{*}(\mathcal{A}), \mathcal{X_{B}}\text{-}\res B''\}\\
&\leq \max\{ \mathcal{X_{B}}\text{-}\res i_{*}(\mathcal{A}), \mathcal{X_{B}}\text{-}\res i_{*}(\mathcal{A})-1\}\\
&\leq \mathcal{X_{B}}\text{-}\res i_{*}(\mathcal{A}).
\end{align*}
So $\mathcal{X_{B}}\text{-}\res j_{!}(\mathcal{X_C})\leq \mathcal{X_B}\text{-}\res i_{*}(\mathcal{A})$.
Also, by (1), we get
 $$\mathcal{X_{B}}\text{-}\res \mathcal{B}\leq \mathcal{X_{C}}\text{-}\res \mathcal{C}+\mathcal{X_{B}}\text{-}\res i_{*}(\mathcal{A})+1.$$

The last assertion follows from (3).

(6)
If $i_{*}(\mathcal{X_{A}}) \subseteq \mathcal{X_{B}}$, then
$\mathcal{X_{B}}\text{-}\res i_{*}(\mathcal{X_{A}})=0$.
The desired assertion follows from (2).

(7)
It follows from (5) and (6).

(8)
If $j^{*}(\mathcal{X_{B}})\subseteq \mathcal{X_{C}}$, then $\mathcal{X_{C}}\text{-}\res j^{*}(\mathcal{X_{B}})=0$.
 So
 $\mathcal{X_{C}}\text{-}\res \mathcal{C}\leq \mathcal{X_{B}}\text{-}\res \mathcal{B}$
 by (4).
\end{proof}

\begin{remark}
One can get the result \cite[Theorem 3.7]{ZHJ1} by applying Theorem \ref{main-gl} to abelian categories.
\end{remark}

Take $\mathcal{X_{A}}=\mathcal{P(A)}$, $\mathcal{X_{B}}=\mathcal{P(B)}$ and $\mathcal{X_{C}}=\mathcal{P(C)}$ in Theorem \ref{main-dim}.
%Notice that $j_{!}(\mathcal{P(C)})\subseteq \mathcal{P(B)}$,
We have

\begin{corollary}{\rm (\cite[Theorem 3.5]{GMT})}\label{main-gl}
Let $(\mathcal{A},\mathcal{B},\mathcal{C})$ be a recollement of extriangulated categories  as the diagram (\ref{recolle}), and let $C\in\mathcal{C}$. Then we have the following statements.
\begin{itemize}
\item[(1)]
\begin{itemize}
\item[(a)] ${\gl \mathcal{B}}\leq \mathcal{P(B)}\text{-}\res i_{*}(\mathcal{A})+\gl \mathcal{C}+1$.
\item[(b)] $\mathcal{P(B)}\text{-}\res i_{*}(\mathcal{A})\leq \gl \mathcal{A}+\sup\{{\pd_{\mathcal{B}}}i_{*}(P)\mid P\in \mathcal{P}(\mathcal{A})\}$.
\item[(c)] $\gl\mathcal{B}\leq \gl \mathcal{A}+\gl\mathcal{C}+\sup\{\pd_{\mathcal{B}}i_{*}(P)\mid P\in \mathcal{P(A)}\}+1.$
\item [(d)] $\gl \mathcal{C}\leq \gl \mathcal{B}+\sup\{{\pd_{\mathcal{C}}}j^{*}(P)\mid P\in \mathcal{P}(\mathcal{B})\}$.
\end{itemize}

\item[(2)] Assume that $i^{!}$ is exact, then
\begin{itemize}
\item[(a)] $\gl\mathcal{B}\leq \gl \mathcal{A}+\gl\mathcal{C}+1.$
\item[(b)] $\gl \mathcal{C}\leq \gl\mathcal{B}.$
\end{itemize}
\end{itemize}

\end{corollary}

\begin{proof}
(1)
(a) Since $j_{!}(\mathcal{P(C)})\subseteq \mathcal{P(B)}$ by Lemma \ref{lem-rec}, the assertion follows from Theorem \ref{main-dim}(5).

(b) It follows from Theorem \ref{main-dim}(2).

(c) It follows from Theorem \ref{main-dim}(5) (or (a) and (b)).

(d) It follows from Theorem \ref{main-dim}(4).

(2)
(a) Since $i^{!}$ is exact by assumption, $i_{*}(\mathcal{P(A)})\subseteq \mathcal{P(B)}$ by Lemma \ref{lem-rec}.
Notice that $j_{!}(\mathcal{P(C)})\subseteq \mathcal{P(B)}$.
The assertion follows from Theorem \ref{main-dim}(7).

(b) Since $i^{!}$ is exact by assumption, $j^{*}(\mathcal{P(B)})\subseteq \mathcal{P(C)}$ by Lemma \ref{lem-rec}.
The assertion follows from Theorem \ref{main-dim}(8).
\end{proof}

\begin{remark}In the sense of \cite{GMT}
$$\mathcal{P_{B}}\text{-}\res i_{*}(\mathcal{A})=\sup\{\pd_{\mathcal{B}}i_{*}(A)\mid A\in \mathcal{A}\}$$
is denoted by $\gl_{\mathcal{A}}\mathcal{B}$.
\end{remark}

Under some conditions, one can take special resolving subcategories $\mathcal{X_A}=\mathcal{X'}$, $\mathcal{X_B}=\mathcal{X}$ and $\mathcal{X_C}=\mathcal{X''}$ of $\mathcal{A}$, $\mathcal{B}$ and $\mathcal{C}$ respectively,
as in Theorem \ref{main-resolving}, then we have

\begin{theorem}\label{dim-=}
Let $(\mathcal{A},\mathcal{B},\mathcal{C})$ be a recollement of extriangulated categories as the diagram (\ref{recolle}).
Assume that $i^{*}$ is exact and $j^{*}$ preserves projective objects, and
assume that $\mathcal{X'}$ and $\mathcal{X''}$ are resolving subcategories of $\mathcal{A}$ and $\mathcal{C}$ respectively,
then
$$\mathcal{X}\text{-}\res\mathcal{B}= \max\{\mathcal{X'}\text{-}\res \mathcal{A},\mathcal{X''}\text{-}\res \mathcal{C}\},$$
where
$\mathcal{X}:=\{B\in \mathcal{B}\mid i^{*}(B)\in \mathcal{X'}\ \text{and} \ j^{*}(B)\in \mathcal{X''}\}$.
\end{theorem}
\begin{proof}
The inequation $\leq$: Assume that  $\mathcal{X'}\text{-}\res \mathcal{A}=m$ and $\mathcal{X''}\text{-}\res \mathcal{C}=n$.
Let $B$ be an object in $\mathcal{B}$.
Since $i^{*}$ is exact,
by Lemma \ref{lem-rec}, there exists an $\mathbb{E}$-triangle
$$\xymatrix@C=15pt{j_{!}j^{*}(B)\ar[r]&B\ar[r]&i_{*}i^{*}(B)
\ar@{-->}[r]&}$$
in $\mathcal{B}$.
Since $i^{*}(B)\in \mathcal{A}$ and $j^{*}(B)\in \mathcal{C}$, we have $\mathcal{X'}\text{-}\res i^{*}(B)\leq m$ and $\mathcal{X''}\text{-}\res j^{*}(B)\leq n$.
Notice that $i_{*}$ is exact and $i_{*}(\mathcal{X'})\subseteq \mathcal{X}$ (by Theorem \ref{main-resolving}), so $\mathcal{X}\text{-}\res i_{*}i^{*}(B)\leq m$.
Since $i^{*}$ is exact, $j_{!}$ is exact by Lemma \ref{lem-rec}.
Notice that $j_{!}(\mathcal{X''})\subseteq \mathcal{X}$ by Theorem \ref{main-resolving}, so
$\mathcal{X}\text{-}\res j_{!}j^{*}(B)\leq n$.
It follows that $\mathcal{X}\text{-}\res \mathcal{B}\leq \max\{m,n\}$ from Lemma \ref{prop-resdim}.

The inequation $\geq$: Assume that $\mathcal{X}\text{-}\res \mathcal{B}=l$, let $A\in \mathcal{A}$. Since $i_{*}(A)\in \mathcal{B}$, we have $\mathcal{X}\text{-}\res i_{*}(A)\leq l$.
Notice that $i^{*}$ is exact by assumption and $i^{*}(\mathcal{X})\subseteq \mathcal{X'}$ by Theorem \ref{main-resolving}, so $\mathcal{X'}\text{-}\res i^{*}i_{*}(A)\leq l$, thus $\mathcal{X'}\text{-}\res A\leq l$ and $\mathcal{X'}\text{-}\res \mathcal{A}\leq l$.
By Theorem \ref{main-dim}(8),
we have $\mathcal{X}\text{-}\res \mathcal{C}\leq l$.
Then
$ \max\{\mathcal{X'}\text{-}\res \mathcal{A},\mathcal{X''}\text{-}\res \mathcal{C}\}\leq \mathcal{X}\text{-}\res\mathcal{B}.$

Thus
$$\mathcal{X}\text{-}\res\mathcal{B}= \max\{\mathcal{X'}\text{-}\res \mathcal{A},\mathcal{X''}\text{-}\res \mathcal{C}\}.$$
\end{proof}

%\begin{remark}
%For special resolving subcategories in $\mathcal{A}$, $\mathcal{B}$ and $\mathcal{C}$, the resolution dimension is identified.
%\end{remark}

In general, $\mathcal{P(B)}\neq\{B\in \mathcal{B}\mid i^{*}(B)\in \mathcal{P(A)}\ \text{and} \ j^{*}(B)\in \mathcal{P(C)}\}$ (see Example \ref{exam}(3)), but we have the following result.

\begin{lemma}\label{lem-pp}
Let $(\mathcal{A},\mathcal{B},\mathcal{C})$ be a recollement of extriangulated categories as the diagram (\ref{recolle}). Then we have the following statements.
\begin{itemize}
\item[(1)] If $j^{*}$ preserves projectives, then
$$\mathcal{P(B)}\subseteq\{B\in \mathcal{B}\mid i^{*}(B)\in \mathcal{P(A)}\ \text{and} \ j^{*}(B)\in \mathcal{P(C)}\}.$$
\item[(2)]
If $i^{*}$ and $i^{!}$ are exact,
then
$$\mathcal{P(B)}=\{B\in \mathcal{B}\mid i^{*}(B)\in \mathcal{P(A)}\ \text{and} \ j^{*}(B)\in \mathcal{P(C)}\}.$$
\end{itemize}
\end{lemma}
\begin{proof}
(1) By Lemma \ref{lem-rec}, we have $i^{*}(\mathcal{P(B)})\subseteq \mathcal{P(A)}$.
Since $j^{*}$ preserves projectives,
$j^{*}(\mathcal{P(B)})\subseteq \mathcal{P(C)}$.
Then
$$\mathcal{P(B)}\subseteq\{B\in \mathcal{B}\mid i^{*}(B)\in \mathcal{P(A)}\ \text{and} \ j^{*}(B)\in \mathcal{P(C)}\}.$$

(2)
Since $i^{!}$ is exact by assumption, $j^{*}$ preserves projectives
by Lemma \ref{lem-rec}.
Then, by (1), we have
$$\mathcal{P(B)}\subseteq\{B\in \mathcal{B}\mid i^{*}(B)\in \mathcal{P(A)}\ \text{and} \ j^{*}(B)\in \mathcal{P(C)}\}.$$
Conversely, let $B\in \mathcal{B}$ such that  $i^{*}(B)\in \mathcal{P(A)}$ {and}  $ j^{*}(B)\in \mathcal{P(C)}$.
Since $i^{*}$ is exact, there exists an $\mathbb{E}$-triangle
\begin{align}\label{E11}
\xymatrix@C=15pt{j_{!}j^{*}(B)\ar[r]&B\ar[r]&i_{*}i^{*}(B)
\ar@{-->}[r]&}
\end{align}
in $\mathcal{B}$.

Since $i^{!}$ is exact, $i_{*}$ preserves projectives by Lemma \ref{lem-rec}.
Notice that $j_{!}$ preserves projectives by Lemma \ref{lem-rec},
so $j_{!}j^{*}(B)$ and $i_{*}i^{*}(B)$ are projective objects in $\mathcal{B}$.
The $\mathbb{E}$-triangle (\ref{E11}) splits, thus $B\cong  j_{!}j^{*}(B)\oplus i_{*}i^{*}(B)\in \mathcal{P(B)}$.
Then $\{B\in \mathcal{B}\mid i^{*}(B)\in \mathcal{P(A)}\ \text{and} \ j^{*}(B)\in \mathcal{P(C)}\}\subseteq\mathcal{P(B)}$, and the desired assertion is obtained.
\end{proof}

\begin{remark}
Let $(\mathcal{A},\mathcal{B},\mathcal{C})$ be a recollement of extriangulated categories as the diagram (\ref{recolle}). Assume that $i^{*}$ is exact and $j^{*}$ preserves projectives. One can easily get
$$\mathcal{P(B)}\subseteq\{B\in \mathcal{B}\mid i^{*}(B)\in \mathcal{P(A)}\ \text{and} \ j^{*}(B)\in \mathcal{P(C)}\}$$
by Theorem \ref{main-resolving}.
\end{remark}

Combining Theorem \ref{dim-=} with Lemma \ref{lem-pp}, we have the following result.
\begin{corollary}
Let $(\mathcal{A},\mathcal{B},\mathcal{C})$ be a recollement of extriangulated categories as the diagram (\ref{recolle}).
If $i^{*}$ and $i^{!}$ are exact,
then $$\gl \mathcal{B}=\max\{\gl \mathcal{A},\gl\mathcal{C}\}.$$
\end{corollary}

\begin{remark}
Let $(\mathcal{A},\mathcal{B},\mathcal{C})$ be a recollement of abelian categories. If $i^{*}$ and $i^{!}$ are exact, then $\mathcal{B}\cong \mathcal{A} \times\mathcal{C}$ (see \cite[Corollary 8.10]{Fr}).
Thus the assertion $\gl \mathcal{B}=\max\{\gl \mathcal{A},\gl \mathcal{C}\}$ is obvious.
\end{remark}

Conversely, we have the following result, which holds true for any subcategory $\mathcal{X}$ of $\mathcal{B}$ ($\mathcal{X}$ is not necessary a resolving subcategory).

\begin{theorem}\label{main-reso-dim-con}
Let $(\mathcal{A},\mathcal{B},\mathcal{C})$ be a recollement of extriangulated categories as the diagram (\ref{recolle}).
Assume that $\mathcal{X}$ is a subcategory of $\mathcal{B}$. Then we have the following statements.
\begin{itemize}
\item[(1)] If $i^{*}$ is exact, then $i^{*}(\mathcal{X})\text{-}\res\mathcal{A}\leq \mathcal{X}\text{-}\res\mathcal{B}$.
\item[(2)] $j^{*}(\mathcal{X})\text{-}\res\mathcal{C}\leq \mathcal{X}\text{-}\res\mathcal{B}$.
\end{itemize}
% In particular,
\end{theorem}
\begin{proof}
Assume that $\mathcal{X}\text{-}\res \mathcal{B}=l$.

(1) Let $A\in \mathcal{A}$. Since $i_{*}(A)\in \mathcal{B}$, we have $\mathcal{X}\text{-}\res i_{*}(A)\leq l$.
Notice that $i^{*}$ is exact by assumption and $i^{*}i_{*}\cong \Id_{\mathcal{A}}$ by Lemma \ref{lem-rec}, so we have $i^{*}(\mathcal{X})\text{-}\res i^{*}i_{*}(A)\leq l$, thus $i^{*}(\mathcal{X})\text{-}\res A\leq l$ and $i^{*}(\mathcal{X})\text{-}\res \mathcal{A}\leq l$, that is, $i^{*}(\mathcal{X})\text{-}\res\mathcal{A}\leq \mathcal{X}\text{-}\res\mathcal{B}$.

(2)
Let $C\in \mathcal{C}$. Since $j_{*}(C)\in \mathcal{B}$, we have $\mathcal{X}\text{-}\res j_{*}(C)\leq l$.
Notice that $j^{*}$ is exact and $j^{*}j_{*}\cong \Id_{\mathcal{C}}$ by Lemma \ref{lem-rec}, so we have $j^{*}(\mathcal{X})\text{-}\res j^{*}j_{*}(C)\leq l$, thus $j^{*}(\mathcal{X})\text{-}\res C\leq l$ and $j^{*}(\mathcal{X})\text{-}\res \mathcal{C}\leq l$, that is, $j^{*}(\mathcal{X})\text{-}\res\mathcal{C}\leq \mathcal{X}\text{-}\res\mathcal{B}$.
%Thus
%$$\mathcal{X}\text{-}\res\mathcal{B}\geq \max\{\mathcal{X'}\text{-}\res \mathcal{A},\mathcal{X''}\text{-}\res \mathcal{C}\},$$
\end{proof}

Taking $\mathcal{X}=\mathcal{P(B)}$ in Theorem \ref{main-reso-dim-con}, then we have
\begin{corollary}
Let $(\mathcal{A},\mathcal{B},\mathcal{C})$ be a recollement of extriangulated categories  as the diagram (\ref{recolle}).
Then we have the following statements.
\begin{itemize}
\item[(1)] If $i^{*}$ is exact, $\gl \mathcal{A}\leq \gl \mathcal{B}$.
\item[(2)] {\rm (cf. Corollary \ref{main-gl}(2)(b))} If $i^{!}$ or $j_{*}$ is exact, then $\gl \mathcal{C}\leq \gl \mathcal{B}$.
\end{itemize}
\end{corollary}
\begin{proof}
(1)
Notice that $i^{*}(\mathcal{P(B)})\subseteq \mathcal{P(A)}$ by Lemma \ref{lem-rec}, so the desired result follows from Theorem \ref{main-reso-dim-con}.

(2)
Notice that if $i^{!}$ or $j_{*}$ is exact, then $ j^{*}(\mathcal{P(B)})\subseteq \mathcal{P(C)}$ by Lemma \ref{lem-rec}, so the desired result follows from Theorem \ref{main-reso-dim-con}.
\end{proof}

\section{Examples}

We give some  examples to illustrate the obtained results.

Let $A, B$ be artin algebras and $_{A}M_{B}$ an $(A,B)$-bimodule, and let $\Lambda={A\ {M}\choose\  0\  \ B}$ be a triangular matrix algebra.
Then any module in $\mod \Lambda$ can be uniquely written as a triple ${X\choose Y}_{f}$ with $X\in\mod A$, $Y\in\mod B$
and $f\in\Hom_{A}(M\otimes_{B}Y,X)$ (see \cite[p.76]{AMRISSO95R} for more details).

\begin{example}\label{exam}
{\rm
 Let $A$ be a finite dimensional algebra given by the quiver $\xymatrix@C=15pt{1\ar[r]&2}$ and $B$ be a finite dimensional algebra given by the quiver $\xymatrix@C=15pt{3\ar[r]^{\alpha}&4\ar[r]^{\beta}&5}$ with the relation $\beta\alpha=0$. Define a triangular matrix algebra $\Lambda={A\ \ A\choose 0\ \ B}$, where the right $B$-module structure on $A$ is induced by the unique algebra surjective homomorphsim $\xymatrix@C=15pt{B\ar[r]^{\phi}&A}$ satisfying $\phi(e_{3})=e_{1}$, $\phi(e_{4})=e_{2}$, $\phi(e_{5})=0$.  Then $\Lambda$ is
a finite dimensional algebra given by the quiver
$$
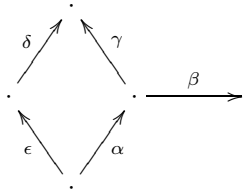
\xymatrix@C=15pt{&\cdot\\
\cdot\ar[ru]^{\delta}&&\ar[lu]_{\gamma}\cdot\ar[rr]^-{\beta}&&\cdot\\
&\ar[lu]^{\epsilon}\cdot\ar[ru]_{\alpha}}$$
with the relations $\gamma\alpha=\delta\epsilon$ and $\beta\alpha=0$. The Auslander-Reiten quiver of $\Lambda$ is
$$\xymatrix@C=15pt{{0\choose P(5)}\ar[rd]&&{S(2)\choose S(4)}\ar[rd]&&{S(1)\choose 0}\ar[rd]&&0\choose P(3)\ar[rd]\\
&{S(2)\choose P(4)}\ar[ru]\ar[rd]&&P(1)\choose S(4)\ar[ru]\ar[r]\ar[rd]&P(1)\choose P(3)\ar[r]&S(1)\choose P(3)\ar[ru]\ar[rd]&&{0\choose S(3)}.\\
S(2)\choose 0\ar[ru]\ar[rd]&&P(1)\choose P(4)\ar[ru]\ar[rd]&&0\choose S(4)\ar[ru]&&S(1)\choose S(3)\ar[ru]\\
&P(1)\choose 0\ar[ru]&&0\choose P(4)\ar[ru]}$$

Following \cite[Example 2.12]{PC14H}, we have that
$$\xymatrix{\mod B\ar[rr]!R|-{i_{*}}&&\ar@<-2ex>[ll]!R|-{i^{*}}
\ar@<2ex>[ll]!R|-{i^{!}}\mod \Lambda
\ar[rr]!L|-{j^{*}}&&\ar@<-2ex>[ll]!R|-{j_{!}}\ar@<2ex>[ll]!R|-{j_{*}}
\mod A}$$
is a recollement of module categories, where
\begin{align*}
& i^{*}({X\choose Y}_{f})=Y, &&i_{*}(Y)={0\choose Y}&i^{!}({X\choose Y}_{f})=\Ker(Y\ra \Hom_{A}(A,X)),\\
& j_{!}(X)={X\choose 0},&&j^{*}({X\choose Y}_{f})=X, &j_{*}(X)={X\choose \Hom_{A}(A,X)}.
\end{align*}

By \cite[Lemma 3.2]{L17G}, we know that $i^{*}$ admits a left adjoint $\widetilde{i^{*}}$
 and
$j_{!}$ admits a left adjoint $\widetilde{j_{!}}$,
where
\begin{align*}
\widetilde{i^{*}}(Y)={A\otimes_{B} Y\choose Y}_{1} ,&&  \widetilde{j_{!}}({X\choose Y}_{f})=\Coker f,
\end{align*}
So $i^{*}$ and $j_{!}$ are exact.
Since $_{A}A$ is projective, $j_{*}$ is exact.
\begin{itemize}

\item[(1)] Take resolving subcategories
\begin{align*}
\mathcal{X'}&=\mathcal{P}(\mod B)=\add(P(5)\oplus P(4)\oplus P(3)),\\
\mathcal{X}&=\add\left({0\choose P(5)} \oplus {S(2)\choose 0}
\oplus {S(2)\choose P(4)}\oplus{P(1)\choose 0} \oplus {S(2)\choose S(4)} \oplus {P(1)\choose P(4)} \oplus {0\choose P(4)}\oplus {P(1)\choose S(4)}\oplus {P(1)\choose P(3)}\right),\\
\mathcal{X''}&=\mathcal{P}(\mod A)=\add (S(2)\oplus  P(1))
\end{align*}
of $\mod B$, $\mod \Lambda$ and $\mod A$ respectively.
Clearly $\mathcal{X''}\text{-}\res \mod A=1$ and $\mathcal{X}\text{-}\res i_{*}(\mod B)=1$.
Then by Theorem \ref{main-dim} (or \cite[Theorem 3.7]{ZHJ1}),
we have the following assertions.
\begin{itemize}
\item[(1.1)] Notice that $j_{!}(\mathcal{X''})= \add ({S(2)\choose 0}
\oplus {P(1)\choose 0})\subseteq \mathcal{X}$. Thus $\mathcal{X}\text{-}\res \mod \Lambda\leq \mathcal{X''}\text{-}\res \mod A+\mathcal{X}\text{-}\res i_{*}(\mod B)+1=1+1+1=3$.
\item[(1.2)]  Notice that $j^{*}(\mathcal{X})= \add ({S(2)\choose 0}
\oplus {P(1)\choose 0})\subseteq \mathcal{X''}$. Thus
$1=\mathcal{X''}\text{-}\res \mod A\leq \mathcal{X}\text{-}\res \mod \Lambda$.
\end{itemize}
Then $1\leq \mathcal{X}\text{-}\res \mod \Lambda\leq 3$.
In fact, $\mathcal{X}\text{-}\res \mod \Lambda=2$.
%增加对第四节前几个结论的解释

\item[(2)]
Take resolving subcategories
\begin{align*}
\mathcal{X'}&=\add(P(5)\oplus P(4)\oplus S(4)\oplus P(3)),\\
\mathcal{X}&=\add \left({0\choose P(5)}\oplus{S(2)\choose P(4)}\oplus{S(2)\choose 0}
\oplus {P(1)\choose 0}\oplus {S(2)\choose S(4)}\oplus{P(1)\choose P(4)}\oplus\right. \\
&~~~~~~~\left.{P(1)\choose S(4)}\oplus {0\choose P(4)}\oplus {S(1)\choose 0}\oplus{P(1)\choose P(3)}\oplus {0\choose S(4)}\oplus {S(1)\choose P(3)}\oplus {0\choose P(3)}\right),\\
\mathcal{X''}&=\add (S(2)\oplus  P(1))
\end{align*}
of $\mod B$, $\mod \Lambda$ and $\mod A$ respectively.
Clearly, $\mathcal{X'}\text{-}\res \mod B=2$ and $\mathcal{X''}\text{-}\res \mod A=1$.
Then by Theorem \ref{main-dim} (or \cite[Theorem 3.7]{ZHJ1}), we have the following assertions.
\begin{itemize}
\item[(2.1)] Notice that $i_{*}(\mathcal{X'})= \add ({0\choose P(5)}
\oplus {0\choose P(4)}\oplus {0\choose S(4)}\oplus {0\choose P(3)})\subseteq \mathcal{X}$. So $\mathcal{X}\text{-}\res i_{*}(\mod B) \leq \mathcal{X'}\text{-}\res \mod B=2 $.
It is clear that $\mathcal{X}\text{-}\res i_{*}(\mod B) =1$.

\item[(2.2)] Notice that $j_{!}(\mathcal{X''})= \add ({S(2)\choose 0}
\oplus {P(1)\choose 0})\subseteq \mathcal{X}$ and.  So $\mathcal{X}\text{-}\res \mod \Lambda \leq \mathcal{X''}\text{-}\res \mod A +\mathcal{X}\text{-}\res i_{*}(\mod A)+1=1+1+1=3$.
\end{itemize}
In fact,
$\mathcal{X}\text{-}\res \mod \Lambda=1$.

\item[(3)] Take resolving subcategories
$$\mathcal{X'}=\mathcal{P}(\mod B)=\add(P(5)\oplus P(4)\oplus P(3)),$$
$$\mathcal{X''}=\mathcal{P}(\mod A)=\add (S(2)\oplus  P(1))$$
of $\mod B$ and $\mod A$ respectively. Then by Theorem \ref{main-resolving},
we get a resolving subcategory
\begin{align*}
\mathcal{X}=\add \left({0\choose P(5)}\oplus{S(2)\choose P(4)}\oplus{S(2)\choose 0}
\oplus {P(1)\choose 0}\oplus{P(1)\choose P(4)}\oplus{0\choose P(4)}\oplus{P(1)\choose P(3)}\oplus {0\choose P(3)}\right)
\end{align*}
 in $\mod \Lambda$. It is obvious that $\mathcal{P}(\mod \Lambda) \subsetneqq   \mathcal{X}$.

Clearly, $\mathcal{X'}\text{-}\res \mod B=2$ and
$\mathcal{X''}\text{-}\res \mod A=1$.
So by Theorem \ref{dim-=}, $\mathcal{X}\text{-}\res \mod \Lambda=2$.

\item[(4)]
Take resolving subcategories
$$\mathcal{X'}=\add(P(5)\oplus P(4)\oplus S(4)\oplus P(3)),$$
$$\mathcal{X''}=\add (S(2)\oplus  P(1))$$
of $\mod B$ and $\mod A$ respectively. Then by Theorem \ref{main-resolving},
we get a resolving subcategory
\begin{align*}
\mathcal{X}=\add \left({0\choose P(5)}\oplus{S(2)\choose P(4)}\oplus{S(2)\choose 0}
\oplus {P(1)\choose 0}\oplus {S(2)\choose S(4)}\oplus{P(1)\choose P(4)}\oplus \right. \\
~~~~~~~~~~~~~~~~~~~~~~\left.{P(1)\choose S(4)}\oplus {0\choose P(4)}\oplus{P(1)\choose P(3)}\oplus {0\choose S(4)}\oplus {0\choose P(3)}\right)
\end{align*}
 in $\mod \Lambda$.
Clearly, $\mathcal{X'}\text{-}\res \mod B=1$ and
$\mathcal{X''}\text{-}\res \mod A=1$.
So by Theorem \ref{dim-=}, $\mathcal{X}\text{-}\res \mod \Lambda=1$.

\item[(5)]
Take a resolving subcategory
\begin{align*}
\mathcal{X}=\add \left({0\choose P(5)}\oplus{S(2)\choose P(4)}\oplus{S(2)\choose 0}
\oplus {P(1)\choose 0}\oplus {S(2)\choose S(4)}\oplus{P(1)\choose P(4)}\oplus \right.\\
~~~~~~~\left.{P(1)\choose S(4)}\oplus {0\choose P(4)}\oplus {S(1)\choose 0}\oplus{P(1)\choose P(3)}\oplus {0\choose S(4)}\oplus {S(1)\choose P(3)}\oplus {0\choose P(3)}\right)
\end{align*}
in $\mod \Lambda$.
Notice that
$i_{*}i^{*}(\mathcal{X})=\add ({0\choose P(5)}\oplus{0\choose P(4)}\oplus{0\choose S(4)}\oplus {0\choose P(3)})\subseteq \mathcal{X}$ and
$j_{!}j^{*}(\mathcal{X})=\add ({S(1)\choose 0}\oplus {P(1)\choose 0}\oplus {S(2)\choose 0})\subseteq\mathcal{X}$.
Then by Theorem \ref{main-resolving-conve}, we have that
$$i^{*}(\mathcal{X})=\add (P(5) \oplus P(4)\oplus S(4)\oplus P(3))$$
is a resolving subcategory in $\mod B$, and
$$j^{*}(\mathcal{X})=\add (S(2) \oplus P(1)\oplus S(1))=\mod A$$ is a resolving subcategory in $\mod A$.

Clearly, $\mathcal{X}\text{-}\res \mod \Lambda=1$.
By Theorem \ref{main-reso-dim-con}, we have
$i^{*}(\mathcal{X})\text{-}\res \mod B\leq 1$ and
$j^{*}(\mathcal{X})\text{-}\res \mod A\leq 1$.
In fact,
$i^{*}(\mathcal{X})\text{-}\res \mod B=1$ and
$j^{*}(\mathcal{X})\text{-}\res \mod A=0$.

\item[(6)] The condition ``$i_{*}i^{*}(\mathcal{X})\subseteq\mathcal{X}$''  is not necessary in Theorem \ref{main-resolving-conve}(1).
Take a resolving subcategory
$$\mathcal{X}=\add\left({0\choose P(5)} \oplus {S(2)\choose 0}
\oplus {S(2)\choose P(4)}\oplus{P(1)\choose 0} \oplus {S(2)\choose S(4)} \oplus {P(1)\choose P(4)} \oplus {0\choose P(4)}\oplus {P(1)\choose P(3)} \right)$$ in $\mod \Lambda$.
Notice that
$$i^{*}(\mathcal{X})=\add (P(5) \oplus P(4)\oplus S(4)\oplus P(3))$$ is a resolving subcategory in $\mod B$.
But
$$i_{*}i^{*}(\mathcal{X})=\add \left({0\choose P(5)}\oplus{0\choose P(4)}\oplus{0\choose S(4)}\oplus {0\choose P(3)}\right)\nsubseteq\mathcal{X}.$$
\end{itemize}}
\end{example}

\vspace{0.5cm}
\textbf{Acknowledgement}.
%The authors thank the referee for the useful suggestions.
The first author was supported by the NSF of China (12001168), the Key Research Project of
Education Department of Henan Province (21A110006) and Henan
University of Engineering (DKJ2019010). The second author
was supported by the NSF of China (11901341, 11971225).

%This work was supported by NSFC (11971225, 11901341, 12001168),
%Henan University of Engineering (DKJ2019010), the Key Research Project of Education Department of Henan Province (21A110006), the project ZR2019QA015 supported by Shandong Provincial Natural Science Foundation.
%The authors thank to
%the referee whose valuable corrections and suggestions have  improved the presentation and organization of
%this article.

\end{document}